\documentclass[11pt,reqno,twoside]{amsart}

\usepackage{graphicx}
\usepackage{lmodern} 
\usepackage[T1]{fontenc}
\usepackage{float}
\usepackage{subfigure}
\usepackage{amsmath,mathdots,amssymb,latexsym,amsthm,mathrsfs,epsfig}
\usepackage{multirow}
\usepackage{breqn}
\usepackage{float}
\usepackage{cite}
\usepackage{enumitem}
\usepackage{environ}
\usepackage[colorlinks=true,linkcolor=blue,citecolor=red,urlcolor=black]{hyperref}
\usepackage{enumitem}
\usepackage{multicol}
\usepackage{mathtools}
\usepackage{dsfont}

\NewEnviron{myequation}{%
	\begin{equation}
		\scalebox{1.1}{$\BODY$}
	\end{equation}
}

\usepackage{pifont}
\usepackage{color}
\usepackage{soul}
\renewcommand{\baselinestretch}{1.2}
\makeatletter
\newcommand{\single}{\let\CS=\@currsize\renewcommand{\baselinestretch}{1.1}\tiny\CS}
\newcommand{\singb}{\let\CS=\@currsize\renewcommand{\baselinestretch}{1}\tiny\CS}
\newcommand{\singa}{\let\CS=\@currsize\renewcommand{\baselinestretch}{1.2}\tiny\CS}
\newcommand{\oneandahalfspacing}{\let\CS=\@currsize\renewcommand{\baselinestretch}{1.5}\tiny\CS}
\newcommand{\singlespacing}{\let\CS=\@currsize\renewcommand{\baselinestretch}{1.6}\large\CS}
\newcommand{\bc}{\begin{center}}
	\newcommand{\ec}{\end{center}}
\newcommand{\be}{\begin{eqnarray}}
	\newcommand{\ee}{\end{eqnarray}}

\newcommand{\Hom}{\operatorname{Hom}}

\newcommand{\diag}{\operatorname{diag}}
\newcommand{\Irr}{\operatorname{Irr}}
\newcommand{\Ind}{\operatorname{Ind}}

\newcommand{\Rep}{\operatorname{Rep}}

\newcommand{\Span}{\operatorname{Span}}

\newcommand{\beano}{\begin{eqnarray*}}
	\newcommand{\eeano}{\end{eqnarray*}}

\newcommand{\ba}{\begin{array}}
	\newcommand{\ea}{\end{array}}

\makeatletter

\newtheorem{theorem}{Theorem}[section]
\newtheorem{corollary}[theorem]{Corollary}
\newtheorem{lemma}[theorem]{Lemma}
\newtheorem{prop}[theorem]{Proposition}
\newtheorem{conj}[theorem]{Conjecture}

\theoremstyle{definition}
\newtheorem{definition}[theorem]{Definition}
\newtheorem{rem}[theorem]{Remark}

\numberwithin{equation}{section}
\DeclareMathOperator{\D}{D}

\DeclareMathOperator{\GL}{GL}
\DeclareMathOperator{\Sp}{Sp}
\DeclareMathOperator{\St}{St}
\DeclareMathOperator{\Nrd}{Nrd}

\DeclareMathOperator{\Par}{P}

\DeclareMathOperator{\G}{G}
\DeclareMathOperator{\Ha}{H}

\DeclareMathOperator{\F}{F}
\DeclareMathOperator{\N}{N}

\DeclareMathOperator{\M}{M}
\DeclareMathOperator{\V}{V}

\DeclareMathOperator{\Ad}{Ad}

\allowdisplaybreaks

\begin{document}
\title[Distinguished representations of $\GL_n(\mathcal D)$]
{On representations of $\GL_n$ distinguished by\\ $\GL_1 \times \GL_{n-1}$ over a quaternion division algebra}


\author[Prem Dagar]{Prem Dagar}
\address{Department of Mathematics, Indian Institute of Science Education and Research Tirupati, India}
\email{dagarprem5@gmail.com, prem@labs.iisertirupati.ac.in}

\author[Hariom Sharma]{Hariom Sharma}
\address{Department of Mathematics, Indian Institute of Technology Bombay, Mumbai 400076, India}
\email{hariomshrma97@gmail.com, hariom@math.iitb.ac.in}

\author[Mahendra Kumar Verma]{Mahendra Kumar Verma}
\address{Department of Mathematics, Indian Institute of Technology Roorkee, Uttarakhand 247667, India}
\email{mahendraverma@ma.iitr.ac.in}


\date{}


\begin{abstract}

Let $\D$ be a quaternion division algebra over a non-Archimedean local field $\F$ of characteristic zero, and let $\G_n=\GL_n(\D)$. Let $\Ha_{1,n-1}$ denote the subgroup of $\G_n$ consisting of block-diagonal matrices of the form $\diag(g_1,g_2)$, where $g_1\in \G_1$ and $g_2\in \G_{n-1}$. In this article, we formulate a conjectural classification of irreducible smooth $\Ha_{1,n-1}$-distinguished representations of $\G_n$ for $n>2$. We prove this conjecture in the cases $n=3$ and $n=4$. When $n=2$, the results are well known due to the contributions by various authors.
\end{abstract}
\makeatletter
\def\@setcopyright{}
\def\@serieslogo{}
\def\serieslogo{}
\makeatother
\maketitle

\section{Introduction}\label{intro}

        \noindent Let $G$ be an $l$-group and $H$ be a closed subgroup of $G$. Let  $(\pi,\V)$ be a representation of $G$ on a complex vector space $\V$ and $\chi$ be a character of $H$. Denote by $\Hom_H(\pi,\chi)$, the space of linear forms $\xi$ on $\V$ such that $\xi(\pi(h)v)=\chi(h)\xi(v)$ for all $h\in H$ and $v\in\V$. The representations $\pi$ of $G$ with $\Hom_H(\pi,\chi)\neq 0$ are called $(H,\chi)$-distinguished, in particular $H$-distinguished if $\chi$ is the trivial character $\mathds{1}$.

        Distinguished representations play a central role in harmonic analysis. We briefly explain the motivation for studying them. Distinguished representations in $\widehat{G}$, the unitary dual of $G$, form the basic building blocks for harmonic analysis on the homogeneous space $H\backslash G$. Indeed, the unitary representation $L^{2}(H\backslash G)$ of $G$ decomposes as a direct integral \[L^{2}(H\backslash G) \;=\; \int_{\widehat{G}}^{\oplus} \pi \, d\mu(\pi),\] and the support of the associated Plancherel measure $\mu$ is contained in the class of $H$-distinguished representations.

         We now formulate the distinction problem in our setting. Let $\F$ be a non-Archimedean local field of characteristic zero with finite residue field, and let $\D$ be a quaternion division algebra over $\F$. For $n \ge 1$, set $\G_n=\GL_n(\D)$. 
         We write $\mathds{1}_n$ for the trivial representation of $\G_n$.
         Let $p$ and $q$ be two nonnegative integers with $p+q=n$.  Let $\Ha_{p,q}$ denote the subgroup of $\G_n$ given by $\left\{\diag(g_1,g_2):\;g_1\in \G_p,\;g_2\in \G_{q}\right\}.$
        
          For $G = \G_n$ and $H = \Ha_{p,q}$, the elements of the space $\Hom_H(\pi,\mathds{1})$ are called local linear periods. In this article, we formulate a conjecture concerning the existence of linear periods with respect to $\Ha_{1,n-1}$ for all irreducible smooth representations of $\G_n$ with $n>2$, and verify it in low-rank cases.

          The problem of existence and uniqueness of such linear periods has been widely investigated. In the case $G=\GL_n(\F)$ and $H=\GL_p(\F)\times\GL_q(\F)$, Jacquet and Rallis~\cite[Theorem~1.1]{jacquet1996uniqueness} established the uniqueness of linear periods. Chen and Sun~\cite[Theorem B]{chen2020uniqueness} subsequently studied the uniqueness of twisted linear periods for $\GL_{2n}(\F)$ with respect to characters $\chi$ of $\GL_n(\F)\times\GL_n(\F)$.
          An interesting problem, therefore, is to characterize irreducible representations of $\GL_n(\F)$ that admit nonzero linear periods with respect to $\GL_p(\F)\times\GL_q(\F)$ with $p+q=n$. Yang \cite[Theorem $1.2$]{yang2022linear} examined the linear periods with respect to $\mathrm{GL}_n(\mathrm{F}) \times \mathrm{GL}_n(\mathrm{F})$ for all unitary representations of $\GL_{2n}(\F)$. Subsequently, the irreducible smooth representations of $\GL_4(\F)$ having a linear period with respect to $\GL_2(\F)\times\GL_2(\F)$  were classified by Sharma \cite[Theorem $1$]{hariom}.

         In the context of quaternion division algebra, Anandavardhanan et al. \cite[Theorem 2.5]{anandavardhanan2024sign} studied the problem of linear periods, providing a proof for their uniqueness. Their result states that if $\pi$ is an irreducible representation of $\G_n$ and  $\Ha_{p,q}$ is a subgroup of $\G_n$ with $|p-q|<2$. Then 
         $$\dim \Hom_{\Ha_{p,q}}(\pi,\mathds{1}) \le 1.$$ 
         Moreover, when $|p-q|\ge 2$, this multiplicity-one result holds for both irreducible unitary representations and generic representations (whose Jacquet-Langland transfer is generic). In fact, if $\pi$ is generic and $|p-q|\ge 2$, then $\dim \Hom_{\Ha_{p,q}}(\pi,\mathds{1}) = 0$. They also proved that, for $n>2$, no discrete series representation of $\G_n$ is $\Ha_{1,n-1}$-distinguished.
         More recently, Lu \cite{Lu26} explored the generalized linear period of an irreducible smooth representation $\pi$ of $\G_n$ with respect to $\Ha_{p,q}$, where $p+q=n$.
         We now aim to study all irreducible smooth  representations of $\G_n$ that admits a linear period with respect to $\Ha_{1,n-1}$. 
        
        To begin, we examine the case \( n = 2 \). We refer the reader to Section~\ref{pre} for any unexplained notation. It has been established that if an irreducible representation \(\pi\) of \(\G_2\) is \(\Ha_{1,1}\)-distinguished then \(\pi \simeq \widetilde{\pi}\) (see \cite[Theorem 3.2]{raghu}). The classification of irreducible representations of $\G_2$ that are $\Ha_{1,1}$-distinguished has been established by {Raghuram~\cite[Theorem 1.1]{raghuram2007restriction} and Gan--Takeda~\cite[Theorem 8.6]{gan}. The list includes the trivial representation $\mathds{1}_2$ and $\chi\,\St_2$, where $\chi$ is a character of $\D^{\ast}$ satisfying $\chi|_{(\D^{\ast})^2}=1$. In addition, for a self-dual representation $\sigma$ of $\D^{\ast}$ with $\dim(\sigma)>1$, exactly one of $\St_2(\sigma)$ and $\Sp_2(\sigma)$ is $\Ha_{1,1}$-distinguished. More precisely, $\St_2(\sigma)$ (resp.\ $\Sp_2(\sigma)$) is $\Ha_{1,1}$-distinguished when the central character $\omega_{\sigma}\circ\Nrd_{\D/\F}$ is non-trivial (resp. trivial)(\cite[Corollary~5]{DP}).
        A supercuspidal representation has linear periods with respect to $\Ha_{1,1}$ if the central character is trivial. Finally, the classification also contains the representation $\sigma\times\widetilde{\sigma}$, where $\sigma$ is an irreducible representation of $\D^{\ast}$. 
        In the case $n>2$, we propose the following conjecture, which is the main result of this article.

\begin{conj}\label{conj}
       An irreducible smooth representation $\theta$ of $\G_n$ with $n>2$ is $\Ha_{1,n-1}$-distinguished if and only if  $\theta$ is either $\mathds{1}_n$ or of the form $\mathds{1}_{n-2} \times \tau,$ where $\tau$ is an infinite-dimensional irreducible representation of $\G_2$ which  is $\Ha_{1,1}$-distinguished.
\end{conj}

As a consequence of Conjecture~\ref{conj}, one can deduce the multiplicity-one property for all irreducible smooth representations. More precisely, we have the following corollary.

\begin{corollary}\label{corr}
    Let $\theta$ be an irreducible smooth representation of $\G_n$. Then 
    $\dim \Hom_{\Ha_{1,n-1}}(\theta, \mathds{1}) \leq 1$. 
\end{corollary}

       In the case $\D=\F$, the above conjecture specializes to \cite[Corollary A.2.2]{Smith17}, which is a reformulation of the results of Prasad \cite[Theorem~2]{Pra93} and Venketasubramanian \cite[Corollary~6.15]{Ven13}. We verify Conjecture \ref{conj} in the cases $n=3$ and $n=4$ (see Theorem \ref{t1} and Theorem \ref{t2}, respectively).

\subsection*{Strategy of the proof}
        We outline the strategy for verifying the Conjecture \ref{conj} for $n=3, 4$. Since irreducible supercuspidal representations of $\G_n$ for $n>2$ are not $\Ha_{1,n-1}$-distinguished, it suffices to consider irreducible non-supercuspidal representations, which arise as quotients of parabolically induced representations of the form $\pi=\pi_1 \times \pi_2$, where $\pi_1 \in \Irr(\G_k)$ and $\pi_2 \in \Irr(\G_{n-k})$ with $1 \le k \le n-1$. Using Mackey theory for the restriction of parabolically induced representations, we first obtain a comprehensive list of representations that may be \(\Ha_{1,n-1}\)-distinguished. These representations are not necessarily irreducible.  We then determine their irreducible subquotients (up to permutation of the inducing data) explicitly by using  the Zelevinsky and Langlands classifications.  Finally, we analyze the \(\Ha_{1,n-1}\)-distinction of the relevant subquotients, which leads to a complete classification of irreducible \(\Ha_{1,n-1}\)-distinguished representations of \(\G_n\).

\subsection*{Organization of the paper} 
            This paper is organized as follows. In Section~\ref{pre}, we fix the notations and recall basic facts on representations of $\G_n$, together with known results on $\Ha_{1,n-1}$-distinguished representations. We also derive the necessary conditions for $\Ha_{1,n-1}$-distinction by using geometric lemma. In Sections~\ref{G_3} and~\ref{G_4}, we study $\Ha_{1,n-1}$-distinguished representations of $\G_n$ in the cases $n=3$ and $n=4$, respectively. The multiplicity-one property for the pair $(\G_n,\Ha_{1,n-1})$ is established in Section~\ref{sec5}.

\section{Notations and Preliminaries}\label{pre}

     \noindent Let \(\F\) be a non-Archimedean local field of characteristic zero and let \(\D\) be a quaternion division algebra over \(\F\) with \(\dim_{\F} \D=4\). For \(n\geq 0\), set \(\G_n=\GL_n(\D)\). We denote by \(\Nrd_{\D/\F}:\G_n\to \F^\times\), the reduced norm map, and define \(\nu_n(g)=|\Nrd_{\D/\F}(g)|_{\F},\ g\in \G_n\). When the index is clear from the context, we simply write \(\nu\) instead of \(\nu_n\). Let \(p,q\geq 0\) with \(p+q=n\), and put \(\delta_{p,q}=\diag(I_p,-I_q)\). Consider the inner automorphism of $\G_n$ given by $$\theta := \Ad(\delta_{p,q}) : g \mapsto \delta_{p,q}\, g \, \delta_{p,q}^{-1}.$$ Its fixed-point subgroup is $\Ha_{p,q}=\left\{\diag(g_1,g_2):\;g_1\in \G_p,\;g_2\in \G_{q}\right\}.$ For \(s\in \mathbb{R}\), let $\alpha_s:\F^\times \rightarrow\mathbb{C}^\times$ denote the character defined by $\alpha_s(x)=|x|_{\F}^{\,2s},~x\in\F^\times.$ Then, we define corresponding character of \(\Ha_{p,q}\), given by
\[
\chi_s(\diag(g_1,g_2))
=
\alpha_s\left(\Nrd_{\D/\F}(g_1)\Nrd_{\D/\F}(g_2)^{-1}\right)
=
\nu(g_1)^{2s}\nu(g_2)^{-2s},
\qquad
g_1\in \G_p,\ g_2\in \G_q.
\]

\subsection{Representations of $\G_n$}
         
          In this section, we review the necessary background on the representation theory of $\G_n$. The results that we recall are well known and our main reference is \cite{bern76}. Let $\Rep(\G_n)$ denote the category of smooth complex representations of $\G_n$ of finite length. We write $\Irr(\G_n)$ for the set of equivalence classes of irreducible representations and $\mathcal{C}(\G_n)$ for the subset of supercuspidal representations. We shall often use the abbreviation s.c. for supercuspidal. If $\pi \in \Rep(\G_n)$, we denote its contragredient by $\widetilde{\pi}$. All representations considered in this article are smooth.

          Throughout this article, we write $\G$ in place of $\G_n$ whenever there is no ambiguity in the value of $n$. Let $\Par$ be a parabolic subgroup of $\G$ with Levi decomposition $\Par = \M \N$, where $\M$ is a Levi subgroup and $\N$ is the unipotent radical. Any representation $\rho$ of $\M$ may be inflated to a representation of $\Par$, still denoted by $\rho$, by letting $\N$ act trivially. We define the normalized parabolically induced representation of $\G$ by $ \Ind_{\Par}^{\G}(\delta_{\Par}^{1/2} \otimes \rho)$, where $\delta_{\Par}$ denote the modular character of $\Par$.  Let $\alpha=(n_1,\ldots,n_r)$ be a partition of $n$ and $\M_{\alpha}\cong\G_{n_1}\times\cdots\times\G_{n_r}$ be the corresponding Levi subgroup of $\G_n$. Then any representation of $\M_{\alpha}$ is of the form $\rho_1\otimes\cdots\otimes\rho_{n_r}$, where $\rho_i\in\Rep({\G_{n_i}})$ for $i=1,\ldots,r$. Define $\rho_1\times\cdots\times\rho_{n_r}$ as the representation $\Ind_{\Par_{\alpha}}^{\G_n}(\rho_1\otimes\cdots\otimes\rho_{n_r})$. 
          
          We now recall the Jacquet functor, left adjoint to the parabolic induction functor $\Ind_{\Par}^{\G}$. Let $(\pi,V)$ be a representation of $\G$. Define the subspace $V(\N) := \Span\{\pi(n)v - v : n \in \N,\ v \in V\}$. Set $V_{\N} := V / V(\N)$. The action of $\Par$ on $V_{\N}$ is given by $\pi_{\N}(p)(v + V(\N)) := \delta_{\Par}^{-1/2}(p)\,\pi(p)v + V(\N)$, for $p \in \Par$. Since the modular character $\delta_{\Par}$ is trivial on $\N$, the representation $(\pi_{\N},V_{\N})$ factors through $\Par/\N$ and may therefore be regarded as a representation of the Levi subgroup $\M \simeq \Par/\N$. The representation $(\pi_{\N},V_{\N})$ is called the normalized Jacquet module of $\pi$. For Levi subgroup $\M=\G_l\times\G_{n-l}$, we denote the Jacquet module of $\pi$ by $r_{(l,n-l)}(\pi)$.

\subsection{Zelevinsky and Langlands classification } 

          The results presented in this subsection follow \cite{tadic1990induced,ming}. For any $\rho\in\mathcal{C}(\G)$, let $s(\rho)$ denote the smallest non-negative real number such that the induced representation $\rho \times \nu^{s(\rho)}\rho$ is reducible. Then, $s(\rho)$ is an integer dividing $2$. We write $\nu_\rho := \nu^{s(\rho)}$. Note that if $\rho \in \mathcal{C}(\G_1)$, then $s(\rho)=2$ when $\dim(\rho)=1$ and $s(\rho)=1$ otherwise, whereas if $\rho' \in \mathcal{C}(\G_2)$, then $s(\rho')=1$. Let $\rho \in \mathcal{C}(\G)$ and $a,b \in \mathbb{Z}$ with $a \le b$. The set $\Delta:= [a,b]_{(\rho)} = \{\nu_\rho^{a}\rho,\, \nu_\rho^{a+1}\rho,\, \ldots,\, \nu_\rho^{b}\rho\}$ is called a segment of irreducible cuspidal representations. We call $a$ and $b$ the beginning and  ending of $\Delta$, respectively. We denote by $l(\Delta) := b-a+1$, the length of  segment $\Delta$. The contragredient of $\Delta$ is defined by $\widetilde{\Delta} := [-b, -a]_{(\widetilde{\rho})}$. 
\begin{definition}
         Let $\Delta_1 = [a_1,b_1]_{(\rho)} $ and $\Delta_2 = [a_2,b_2]_{(\rho)} $ be two segments. We say that $\Delta_1$ and $\Delta_2$ are linked if $\Delta_1 \nsubseteq \Delta_2$, $\Delta_2 \nsubseteq \Delta_1$, and $\Delta_1 \cup \Delta_2$ is also a segment.  We say that $\Delta_1$ precedes $\Delta_2$ if $a_1<a_2$, $b_1<b_2$, and $a_2 \le b_1+1$. 
\end{definition}

        Segments provide a systematic way to construct new representations of $\G$, forming the basis of the Zelevinsky and Langlands classifications. To each segment $\Delta = [a,b]_{(\rho)}$, we associate an irreducible representation $Z(\Delta)$ (resp. $L(\Delta)$), defined as the unique irreducible subrepresentation (resp. quotient) of $\nu_\rho^{a}\rho \times \nu_\rho^{a+1}\rho \times \cdots \times \nu_\rho^{b}\rho$. For instance, if $\Delta=\big[-\frac{(n-1)}{2},\frac{n-1}{2}\big]_{(\mathds{1}_1)}$, we have $Z(\Delta)=\mathds{1}_n$, the trivial representation of $\G_n$, and $L(\Delta)=\St_n$, the Steinberg representation of $\G_n$. While for $\Delta = \big[-\frac{(n-1)}{2}, \frac{n-1}{2}\big]_{(\sigma)}$ with $\sigma \in \Irr(\G_1)$ and $\dim(\sigma)>1$, we obtain $Z(\Delta)=\Sp_n(\sigma)$, the Speh representation of $\G_n$, and $L(\Delta)=\St_n(\sigma)$, the generalized Steinberg representation of $\G_n$.

        Let $\mathcal{S}$ denote the set of all segments in $\mathcal{C}(\G)$, and let $\mathcal{M}(\mathcal{S})$ be the set of all finite multi sets of elements. Given a multi-set of segments $\mathfrak{a} = \{\Delta_1, \ldots, \Delta_s\}$, we define $\lambda(\mathfrak{a}) := Z(\Delta_1) \times \cdots \times Z(\Delta_s)$. If for every $i<j$, the segment $\Delta_i$ does not precede $\Delta_j$, then $Z(\Delta_1,\ldots,\Delta_s)$ denotes the unique irreducible subrepresentation of $\lambda(\mathfrak{a})$, independent of the ordering of the segments. Similarly, setting $\pi(\mathfrak{a}) := L(\Delta_1) \times \cdots \times L(\Delta_s)$, the unique irreducible quotient of $\pi(\mathfrak{a})$ is denoted by $L(\Delta_1,\ldots,\Delta_s)$. 
        
        An elementary operation on $\mathfrak{a}\in \mathcal{M}(\mathcal{S})$ consists of replacing a pair of linked segments $\{\Delta_1,\Delta_2\}$ in $\mathfrak{a}$ with the pair $\{\Delta_1\cup\Delta_2,\;\Delta_1\cap\Delta_2\}$. We define a partial order on $\mathcal{M}(\mathcal{S})$ by declaring $\mathfrak{b} \le \mathfrak{a}$ if $\mathfrak{b}$ can be obtained from $\mathfrak{a}$ through a finite sequence of elementary operations. Now, we recall some results from \cite{ming} that are key tools for the explicit analysis of subquotients of representations of the form $\lambda(\mathfrak{a})$ of $\G_n$.

\begin{lemma}[{\cite[Lemma 5.12]{ming}}]\label{l1}
        Let $\Delta_1$ and $\Delta_2$ be two segments, and set $\pi = Z(\Delta_1)\times Z(\Delta_2)$.

        \noindent\upshape(1) The representation $\pi$ is irreducible if and only if $\Delta_1$ and $\Delta_2$ are not linked. 

        \noindent\upshape(2) If $\Delta_1$ and $\Delta_2$ are linked, then $\pi$ has length $2$. Furthermore, when $\Delta_2$ precedes $\Delta_1$, then $\pi$ admits a unique irreducible subrepresentation $Z(\Delta_1,\Delta_2)$ and a unique irreducible quotient $Z(\Delta_1 \cup \Delta_2)\times Z(\Delta_1 \cap \Delta_2)$. When $\Delta_1$ precedes $\Delta_2$, then $\pi$ admits a unique irreducible subrepresentation $Z(\Delta_1 \cup \Delta_2)\times Z(\Delta_1 \cap \Delta_2)$ and a unique irreducible quotient $Z(\Delta_1,\Delta_2)$.
\end{lemma}

\begin{lemma}[{\cite[Corollary~5.15]{ming}}]\label{l2}
       For $\mathfrak{a}, \mathfrak{b} \in \mathcal{M}(\mathcal{S})$, the representation $Z(\mathfrak{b})$ occurs as a subquotient of $\lambda(\mathfrak{a})$ if and only if $\mathfrak{b} \le \mathfrak{a}$.
\end{lemma}

        Next, we recall some results from \cite{tadic1990induced} that serve as key tools for the explicit analysis of subquotients of representations of the form $\pi(\mathfrak{a})$ of $\G_n$.

\begin{lemma}[{\cite[Proposition~4.3]{tadic1990induced}}]\label{l3}
        Let $\Delta_1$ and $\Delta_2$ be two segments and set $\pi = L(\Delta_1)\times L(\Delta_2)$. 

        \noindent\upshape(1) The representation $\pi$ is irreducible if and only if $\Delta_1$ and $\Delta_2$ are not linked. 

        \noindent\upshape(2) If $\Delta_1$ and $\Delta_2$ are linked, then $\pi$ has length $2$. Moreover,  when $\Delta_2$ precedes $\Delta_1$, then $\pi$ admits a unique irreducible subrepresentation $L(\Delta_1 \cup \Delta_2)\times L(\Delta_1 \cap \Delta_2)$ and a unique irreducible quotient $L(\Delta_1,\Delta_2)$. When $\Delta_1$ precedes $\Delta_2$, then $\pi$ admits a unique irreducible subrepresentation $L(\Delta_1,\Delta_2)$ and a unique irreducible quotient $L(\Delta_1 \cup \Delta_2)\times L(\Delta_1 \cap \Delta_2)$.
\end{lemma}

\begin{lemma}[{\cite[Proposition~4.4]{tadic1990induced}}]\label{l4}
          For $\mathfrak{a}, \mathfrak{b} \in \mathcal{M}(\mathcal{S})$, the representation $L(\mathfrak{b})$ occurs as a subquotient of $\pi(\mathfrak{a})$ if and only if $\mathfrak{b} \le \mathfrak{a}$.
\end{lemma}

\subsection{M{\oe}glin-Waldspurger algorithm }\label{walds}
    
          In this subsection, we recall M{\oe}glin-Waldspurger algorithm.  Badulescu and Renard~\cite[\S\,1]{BR07} extend the M{\oe}glin--Waldspurger algorithm (see \cite[II.1 and II.2]{MW89}) for computing the Zelevinsky dual to the Langlands framework for irreducible representations of $\G_n$. The M{\oe}glin--Waldspurger algorithm associates to a multisegment $\mathfrak{a}$, a multisegment $\mathfrak{a}^t$ such that $Z(\mathfrak{a}) \simeq L(\mathfrak{a}^t)$, and conversely which we summarize below.

          If $\Delta=[a,b]_{(\rho)}$ is a segment, we set $\Delta^-=[a,,\ldots,b-1]_{(\rho)}$, with the convention that $\Delta^-$ is void when $a=b$. Let $\Delta_1=[a_1,b_1]_{(\rho)}$ and $\Delta_2=[a_2,b_2]_{(\rho)}$ be two segments. We  write $\Delta_1 \ge \Delta_2$ if $a_1>a_2$, or if $a_1=a_2$ and $b_1\ge b_2$. This defines a total order on the set of segments. A multisegment $\mathfrak{a}=(\Delta_1,\Delta_2,\ldots,\Delta_t)$ is said to be ordered if $\Delta_1 \ge \Delta_2 \ge \cdots \ge \Delta_t$. The lexicographic order induces a total order on ordered multisegments: if $\mathfrak{a}=(\Delta_1,\ldots,\Delta_t)$ and $\mathfrak{a}'=(\Delta'_1,\ldots,\Delta'_{t'})$ are ordered multisegments, then $\mathfrak{a}\ge \mathfrak{a}'$ if $\Delta_1>\Delta'_1$, or $\Delta_1=\Delta'_1$ and $\Delta_2>\Delta'_2$, and so on, or if $t\ge t'$ and $\Delta_i=\Delta'_i$ for all $1\le i\le t'$.

         Let $\mathfrak{a}$ be a multisegment. We associate to $\mathfrak{a}$ a multisegment $\mathfrak{a}^t$ as follows. Let $d$ be the largest ending of a segment in $\mathfrak{a}$. Choose a segment $\Delta_{i_0}$ in $\mathfrak{a}$ containing $d$ and maximal with this property. Define integers $i_1,i_2,\ldots,i_r$ inductively so that $\Delta_{i_s}$ is a segment of $\mathfrak{a}$ preceding $\Delta_{i_{s-1}}$, with ending $d-s$, maximal with these properties, and $r$ is maximal for which this construction is possible. Set $\mathfrak{a}^-=(\Delta_1^-,\Delta_2^-,\ldots,\Delta_t^-)$, omitting void segments. For each index $i$, define $\Delta_i'=\Delta_i$ if $i \notin \{i_0,i_1,\ldots,i_r\}$, and $\Delta_i'=\Delta_i^-$ if $i \in \{i_0,i_1,\ldots,i_r\}$. Then the segment $\{d-r,d-r+1,\ldots,d\}$ is the first segment of $\mathfrak{a}^t$. Repeating for $\mathfrak{a}^-$ the same procedure used for $\mathfrak{a}$, we obtain the next segment of $\mathfrak{a}^t$, and continuing inductively we see that $\mathfrak{a}^t$ is the multiset union of the segment $\{d-r,d-r+1,\ldots,d\}$ and $(\mathfrak{a}^-)^t$. The multisegment $\mathfrak{a}^t$ is independent of the choices made in the construction. Moreover, the map $\mathfrak{a}\mapsto \mathfrak{a}^t$ is an involution on the set of non-void multisegments.

\subsection{Some known results on $ \Ha_{1,n-1}$-distinction}  
         In this subsection, we recall several results related to $\Ha_{1,n-1}$-distinction. We begin with the following lemma, whose proof is parallel to that of \cite[Lemma~3.1]{yang2022linear}.

\begin{lemma}\label{cont} If $\pi \in \Irr(\G_n)$ is $\Ha_{1,n-1}$-distinguished, then $\widetilde{\pi}$ is also $\Ha_{1,n-1}$-distinguished.

\end{lemma}

Next, we recall the results of Anandavardhanan et al.~\cite[Theorem 3.8 and Theorem C.3]{anandavardhanan2024sign} concerning the $\Ha_{1,n-1}$-distinction of discrete series representations and their products.
\begin{theorem}\label{anand}
\noindent\upshape(1) For $n>2$, no discrete series representation of $\G_n$ is $\Ha_{1,n-1}$-distinguished.
    
  \noindent\upshape(2) Let $\pi_1$ and $\pi_2$ be two discrete series representations such that $\pi_1\times\pi_2$ is an irreducible representation of $\G_n$ with $n\geq4$. Then $\pi_1\times\pi_2$ is not $\Ha_{1,n-1}$-distinguished.
\end{theorem}
\subsection{Consequence of Geometric Lemma}\label{orbit}
         Let $\G_n=\GL_n(\D)$. For $1\leq k\leq n-1$, let $\Par_{k,n-k}$ denote the standard parabolic subgroup of $\G_n$ corresponding to the partition $(k,n-k)$. 3. In this subsection, we derive necessary conditions for the $\Ha_{1,n-1}$-distinction of $\Ind_{\Par_{k,n-k}}^{\G_n}(\pi)$ via geometric lemma. For this, we need the orbit decomposition of $\Par_{k,n-k}\backslash \G_n/\Ha_{1,n-1}$, whose details may be found in \cite[Section $4$]{yang2022linear} and \cite[\S3.2]{anandavardhanan2024sign}. This decomposition consists of three orbits and using these orbit decompositions we recall the following result from \cite[Corollary 5.2]{yang2022linear}.

\begin{lemma}\label{lem3}
  Let $\pi_1 \in \Rep(\G_k)$ and $\pi_2 \in \Rep(\G_{n-k})$, with $1\leq k\leq n-1$. Suppose the representation $\pi_1 \times \pi_2$ is $\Ha_{1,n-1}$-distinguished. Then at least one of the following conditions holds:
	\begin{enumerate}

    \item[\upshape(1)] $\pi_1=\nu^{n-k-2}$ and $\pi_2$ is $(\Ha_{1,\;n-k-1},\chi_{k/2})\text{-distinguished}$ \textnormal{(closed orbit (I))}.
     
      \item[\upshape(2)]  $\pi_1 \text{ is } (\Ha_{1,\;k-1}, \chi_{(k-n)/2})\text{-distinguished}$  and $\pi_2=\nu^{-(k-2)}$   \textnormal{(closed orbit (II))}.

    \item[\upshape(3)]  $\Hom_{L_{n,k}} ( r_{(k-1,1)}(\pi_1) \otimes r_{(1,n-k-1)}(\pi_2), \, \nu^{n-k-1} \otimes \mathds{1}_{\Delta(\G_1)}
        \otimes \nu^{-(k-1)}) \neq 0$  \textnormal{(orbit (III))}.
	\end{enumerate}
    Here, $L_{n,k} = \G_{k-1} \times \Delta(\G_1) \times \G_{n-k-1}$ with $\Delta(\G_1)=\{(g,g): g\in\G_1\}\subset \G_1\times\G_1.$ The character $\nu^{n-k-1} \otimes \mathds{1}_{\Delta\G_1}
        \otimes \nu^{-(k-1)}$ of $L_{n,k}$ in $(3)$ is given by $(a,g,g,b)\mapsto \nu^{n-k-1}(a)\nu^{-(k-1)}(b)$. Therefore, the contribution from orbit $(\mathrm{III})$ can arise only through irreducible subquotients of $r_{(k-1,1)}(\pi_1)\otimes r_{(1,n-k-1)}(\pi_2)$ of the form $\nu^{n-k-1}\otimes \rho\otimes \widetilde{\rho}\otimes \nu^{-(k-1)}.$
\end{lemma}
\begin{rem}\label{rem1} 
 If $\pi_1\times\pi_2$ satisfies the conditions of closed orbits~(I) and~(II), then $\pi_1\times\pi_2$ is $\Ha_{1,n-1}$-distinguished. 

\end{rem}

\section{Proof of Conjecture \ref{conj}: Case \texorpdfstring{$n=3$}{n=3}}\label{G_3}

       \noindent In this section, we verify Conjecture~\ref{conj} in the case $n=3$. Since irreducible supercuspidal representations of \(\G_3\) are not \(\Ha_{1,2}\)-distinguished by Theorem~\ref{anand}, it suffices to consider non-supercuspidal representations. 
       Any irreducible non-supercuspidal \(\Ha_{1,2}\)-distinguished representation of \(\G_3\) occurs as a quotient of a representation of the form \(\pi_1\times \pi_2\) or \(\pi_2\times \pi_1\), where \(\pi_1\in \Irr(\G_1)\) and \(\pi_2\in \Irr(\G_2)\). Moreover, if a representation admits an \(\Ha_{1,2}\)-distinguished quotient, then the representation itself is \(\Ha_{1,2}\)-distinguished. Therefore, it is enough to analyze the \(\Ha_{1,2}\)-distinguished subquotients of representations of the form $\pi=\pi_1\times \pi_2,~\pi_1\in \Irr(\G_1),~\pi_2\in \Irr(\G_2),$ with \(\pi\) itself being \(\Ha_{1,2}\)-distinguished. Next, we identify all possible choices of \(\pi\) for which \(\Ha_{1,2}\)-distinction occurs.
  
       \begin{rem}
          By the classification of irreducible representations of \(\G_2\), \(\pi_2\) is one of the following: a supercuspidal representation, a character, a twist of the Steinberg representation, a Speh representation, a generalized Steinberg representation, or a principal series representation. The irreducible representations of $\G_2$ that are $\Ha_{1,1}$-distinguished have been classified in the introduction. Note that  \((\Ha_{1,1},\chi_s)\)-distinction is equivalent to \(\Ha_{1,1}\)-distinction for every infinite-dimensional irreducible representation of \(\G_2\).
\end{rem}

      Based on the possible types of the representation $\pi_2$, we apply Lemma~\ref{lem3}  with $n=3$ and $k=1$ to the corresponding representation $\pi$. This allows us to describe the possible structures of $\pi$, arising from the three orbits.
      The closed orbit~(I) yields $\mathds{1}_1 \times \pi_2$, where $\pi_2$ is infinite-dimensional and $(\Ha_{1,1},\chi_{\frac12})$-distinguished.  The closed orbit~(II) gives representations of the form $\nu^{-1}_{\mathds{1}_1} \times Z([0,1]_{(\mathds{1}_1)})$. The third orbit yields representations of the form $\nu_{\mathds{1}_1}\times Z([-1,0]_{(\mathds{1}_1)})$, $\nu^{-1}_{\mathds{1}_1}\times L([0,1]_{(\mathds{1}_1)})$, $\widetilde{\sigma}\times\sigma\times\mathds{1}_1$, and $\sigma\times\mathds{1}_1\times\widetilde{\sigma}$, where $\sigma\in\Irr(\G_1)$ with $\sigma\not\simeq\nu_{\mathds{1}_1}^{\pm1}$, as analyzed in the table below.
      
        \begin{table}[h]
    \centering
    \begin{tabular}{|c|c|c|}
    \hline
    $\pi_2$ & $\pi_1 \otimes r_{(1,1)}(\pi_2)$ & $\pi=\pi_1\times\pi_2$\\ \hline
    $\text{s.c.}$ & No contribution & No representation \\ \hline
    
    $\chi$ & $\pi_1\otimes\nu^{-1}\chi \otimes\nu^{}\chi$ & $\nu_{\mathds{1}_1}\times Z([-1,0]_{(\mathds{1}_1)})$ \\ \hline
    
    $\chi\St_2$ &$ \pi_1\otimes\nu^{}\chi \otimes\nu^{-1}\chi$ & $ \nu_{\mathds{1}_1}^{-1}\times L([0,1]_{(\mathds{1}_1)})$\\ \hline

     $\Sp_2(\sigma)$ &$\pi_1\otimes\nu^{-\frac12}\sigma \otimes\nu^{\frac12}\sigma$ & No representation \\ \hline
    
    $\St_2(\sigma)$  & $\pi_1\otimes\nu^{\frac12}\sigma \otimes\nu^{-\frac12}\sigma$ & No representation \\ \hline
    
 $\sigma_1\times\sigma_2$ & $\pi_1 \otimes (\sigma_1\otimes\sigma_2\bigoplus \sigma_2\otimes\sigma_1)$ & $\widetilde{\sigma}\times\sigma\times\mathds{1}_1$,  $\sigma \times\mathds{1}_1\times \widetilde{\sigma}~ \text{with}~\sigma\not\simeq\nu_{\mathds{1}_1}^{\pm 1}$
    \\ \hline
    \end{tabular}
    \vspace{5mm}
    \caption{Contribution of orbit (III). }
    \label{tab:placeholder_label}
\end{table}  The preceding analysis leads to the following proposition.

\begin{prop}\label{p1}
Let $\theta\in\Irr(\G_3)$ be $\Ha_{1,2}$-distinguished. Then $\theta$ is a subquotient of one of the following representations $\pi$ of $\G_3:$
\begin{enumerate}
    \item[\upshape(1)]
    $\pi =  \nu_{\mathds{1}_1}^{-1}\times L([0,1]_{(\mathds{1}_1)})$.

    \item[\upshape(2)]
    $\pi =  \nu_{\mathds{1}_1}^{-1}\times Z([0,1]_{(\mathds{1}_1)})$.

    \item[\upshape(3)]
    $\pi =  \nu_{\mathds{1}_1}\times Z([-1,0]_{(\mathds{1}_1)})$.

    \item[\upshape(4)]
    $\pi =  \mathds{1}_1\times\sigma\times\widetilde{\sigma}$, where $\sigma\in\Irr(\G_1)~ \text{with}~\sigma\not\simeq\nu_{\mathds{1}_1}^{\pm 1}$.

    \item[\upshape(5)]
    $\pi = \mathds{1}_1\times \tau$, where $\tau\in\Irr(\G_2)$ is infinite-dimensional and $\Ha_{1,1}$-distinguished.
\end{enumerate}
\end{prop}

       Next, we prove Conjecture~\ref{conj} for the case \(n=3\), which we state below for convenience.

\begin{theorem}\label{t1}

          Let $\theta\in\Irr(\G_3)$. Then $\theta$ is $\Ha_{1,2}$-distinguished if and only if it is either $\mathds{1}_3$, or of the form $\mathds{1}_1 \times \tau$, where $\tau \in \Irr(\G_2)$ is infinite-dimensional and $\Ha_{1,1}$-distinguished.
          
\end{theorem}

\begin{proof}

        The proof proceeds by examining each representation appearing in Proposition~\ref{p1}, and determining the $\Ha_{1,2}$-distinguished irreducible subquotients in each case.

\smallskip

        \noindent \textbf{Case (1).} $\pi=\nu_{\mathds{1}_1}^{-1}\times L([0,1]_{(\mathds{1}_1)})$. In this case, it follows from Lemma~\ref{l3} that $\pi$ has two irreducible subquotients $\theta_1 = L([-1,1]_{(\mathds{1}_1)})$ and $\theta_2 = L([-1]_{(\mathds{1}_1)},[0,1]_{(\mathds{1}_1)})$. Since $\theta_1$ is a discrete series representation of $\G_3$, it is not $\Ha_{1,2}$-distinguished by Theorem~\ref{anand}. It is easy to observe that $\widetilde{\theta_2}$ is isomorphic to $L([1]_{(\mathds{1}_1)},[-1,0]_{(\mathds{1}_1)})$, and therefore it can be realized as a unique irreducible quotient of $\nu_{\mathds{1}_1} \times L([-1,0]_{(\mathds{1}_1)})$. Applying Lemma \ref{lem3} with $n=3$ and $k=1$, we deduce that the latter representation is not $\Ha_{1,2}$-distinguished. Hence, $\widetilde{\theta_2}$ is not $\Ha_{1,2}$-distinguished. Thus, Lemma \ref{cont} implies that $\theta_2$ is also not $\Ha_{1,2}$-distinguished. 

\smallskip

         \noindent \textbf{Case (2).}  $\pi=\nu_{\mathds{1}_1}^{-1}\times Z([0,1]_{(\mathds{1}_1)})$. In this context, Lemma~\ref{l1} implies that $\pi$ is glued from two irreducible subquotients $\theta_1 = Z([-1,1]_{(\mathds{1}_1)})$ and $\theta_2 = Z([-1]_{(\mathds{1}_1)}, [0,1]_{(\mathds{1}_1)})$. Of these, $\theta_1$ is the trivial character, and consequently $\Ha_{1,2}$-distinguished. By M{\oe}glin-Waldspurger algorithm $\S$\ref{walds}, $\theta_2$ is isomorphic to $L([1]_{(\mathds{1}_1)},[-1,0]_{(\mathds{1}_1)})$. The question of $\Ha_{1,2}$-distinction for the latter representation has already been addressed in the previous case.

\smallskip

         \noindent \textbf{Case (3).} $\pi=\nu_{\mathds{1}_1}\times Z([-1,0]_{(\mathds{1})})$. In this situation, applying Lemma~\ref{l1}, we get that $\pi$ has two irreducible subquotients $Z([-1,1]_{(\mathds{1}_1)})$ and $Z([1]_{(\mathds{1}_1)},[-1,0]_{(\mathds{1}_1)})$. It is easy to observe that the question of $\Ha_{1,2}$-distinction for both subquotients has already been discussed in previous cases. 

\smallskip
         \noindent \textbf{Case (4).} 
          $\pi=\mathds{1}_1\times \sigma \times \widetilde{\sigma}$, where $\sigma\in\Irr(\G_1)$ with $\sigma\not\simeq\nu_{\mathds{1}_1}^{\pm 1}$. Suppose $\sigma$ is not of the following forms:
    \begin{itemize}
          \item $\nu_{\chi}^{\pm \frac{1}{2}}\chi$ with $\chi$ being a square trivial character,
    
          \item $\nu_{\sigma'}^{\pm\frac{1}{2}}\sigma'$, where $\sigma'$ is self-dual representation with $\dim(\sigma')>1$.
    \end{itemize} 
          Then $\pi$ is irreducible and is $\Ha_{1,2}$-distinguished by Remark \ref{rem1} because  it satisfies the condition of closed orbit~(I). 
 So, assume otherwise. Let us divide our analysis into two subcases.

        \textbf{(4a)} $\sigma=\nu_{\chi}^{\pm \frac{1}{2}}\chi$. In this case, $\pi$ is reducible and it follows from Lemma~\ref{l1} that it has two irreducible subquotients $\theta_1=\mathds{1}_1\times\chi\St_2$ and $\theta_2=\mathds{1}_1\times\chi$. Of these, $\Ha_{1,2}$-distinction of  $\theta_1$ follows from the closed orbit (I) (see Remark \ref{rem1}). On the other hand, $\theta_2$ is not $\Ha_{1,2}$-distinguished because it does not satisfy any of the conditions of Lemma~\ref{lem3} for the situation $n=3$ and $k=1$.

\smallskip

     \textbf{(4b)}  $\sigma=\nu_{\sigma'}^{\pm\frac{1}{2}}\sigma'$. In this scenario, $\pi$ is reducible, and by Lemma~\ref{l1} we obtain  that it is glued from two irreducible subquotients $\theta_1=\mathds{1}_1\times\St_2(\sigma')$ and $\theta_2=\mathds{1}_1\times\Sp_2(\sigma')$. If follows from Lemma \ref{lem3} with $n=3$ and $k=1$ that $\theta_1$ (resp. $\theta_2$) is $\Ha_{1,2}$-distinguished if and only if the central character $\omega_{\sigma'}$ is non-trivial (resp. trivial).

\smallskip

        \noindent \textbf{Case (5).} $\pi=\mathds{1}_1\times \tau$, where $\tau\in\Irr(\G_2)$ is infinite-dimensional and $\Ha_{1,1}$-distinguished. Here, if $\tau\not\simeq\nu_{\mathds{1}_1}\times\nu_{\mathds{1}_1}^{-1}$, then $\pi$ is irreducible, and its $\Ha_{1,2}$-distinction follows from the closed orbit (I) (see Remark \ref{rem1}).

        On the other hand, if $\tau=\nu_{\mathds{1}_1}\times\nu_{\mathds{1}_1}^{-1}$, then $\pi$ is reducible, and Lemma~\ref{l2} shows that it decomposes into four irreducible subquotients $\theta_1=Z([-1,1]_{(\mathds{1}_1)})$, $\theta_2=Z([-1]_{(\mathds{1}_1)},[0,1]_{(\mathds{1}_1)})$, $\theta_3=Z([-1,0]_{(\mathds{1}_1)},[1]_{(\mathds{1}_1)})$, and $\theta_4=Z([-1]_{(\mathds{1}_1)},[0]_{(\mathds{1}_1)},[1]_{(\mathds{1}_1)})$. 
        Using the The M{\oe}glin--Waldspurger algorithm ~\S\ref{walds}, one obtains $\theta_3 \simeq L([-1]_{(\mathds{1}_1)},[0,1]_{(\mathds{1}_1)})$ and $\theta_4 \simeq L([-1,1]_{(\mathds{1}_1)})$.
        The question of $\Ha_{1,2}$-distinction for $\theta_1$ and $\theta_2$ (resp. $\theta_3$ and $\theta_4$) has been discussed in Case $2$ (resp. Case $1$) of this proof.

  Next, we prove the converse of the theorem. If $\theta$ is trivial, then it is $\Ha_{1,2}$-distinguished by definition. In the case where  $\theta=\mathds{1}_1\times \tau$ with $\tau\in\Irr(\G_2)$ being infinite-dimensional and $\Ha_{1,1}$-distinguished, Lemma \ref{lem3} with $n=3$ and $k=1$ implies that $\theta$ admits $\Ha_{1,2}$-distinction arising from the closed orbit (I). 
\end{proof}

\section{Proof of Conjecture \ref{conj}: Case \texorpdfstring{$n=4$}{n=4}}\label{G_4}

        \noindent This section proves Conjecture~\ref{conj} in the case $n=4$. By Theorem~\ref{anand}, an irreducible supercuspidal representation of $\G_4$ can not be $\Ha_{1,3}$-distinguished. Therefore, it remains to consider non-supercuspidal representations. In this direction, we have the following result, whose  proof is similar to ~\cite[Lemma~17]{hariom}.
        \begin{lemma}\label{hari}
            Let $\theta$ be an irreducible non-supercuspidal representation of $\G_4$ that is $\Ha_{1,3}$-distinguished. Then $\theta$ occurs as a quotient of a representation of the form $\pi=\pi_1 \times \pi_2$, where either $\pi_1 \in \mathcal{C}(\G_1)$ and $\pi_2 \in \mathcal{C}(\G_3)$, or $\pi_1,\pi_2 \in \Irr(\G_2)$. 
        \end{lemma}
        
        In view of the preceding lemma, it remains to study \(\Ha_{1,3}\)-distinction for irreducible quotients of representations of the form \(\pi\) appearing there. Since $\pi$ admits an $\Ha_{1,3}$-distinguished quotient, it itself is $\Ha_{1,3}$-distinguished. To this end, we analyze all possible choices of \(\pi\) and, using Lemma~\ref{lem3}, determine the cases in which \(\Ha_{1,3}\)-distinction can occur. This is formulated in the following proposition.

\begin{prop}\label{p2}
Let \(\theta\in\Irr(\G_4)\) be \(\Ha_{1,3}\)-distinguished. Then \(\theta\) is a subquotient of one of the following representations \(\pi\) of \(\G_4\):
\begin{enumerate}
    \item[\upshape(1)]
    \(\pi=\St_2 \times \St_2\).

    \item[\upshape(2)]
    \(
    \pi=
    Z\!\left(\left[\frac{1}{2},\frac{3}{2}\right]_{(\mathds{1}_1)}\right)
    \times
    Z\!\left(\left[-\frac{3}{2},-\frac{1}{2}\right]_{(\mathds{1}_1)}\right).
    \)

    \item[\upshape(3)]
    \(
    \pi=
    \nu^{\frac{1}{2}}_{\mathds{1}_1}
    \times
    \mu
    \times
    \widetilde{\mu}
    \times
    \nu^{-\frac{1}{2}}_{\mathds{1}_1},
    \)
    where \(\mu\in\Irr(\G_1)\) with
    \(
    \mu \notin
    \left\{
    \nu^{-\frac{1}{2}}_{\mathds{1}_1},
    \nu^{\frac32}_{\mathds{1}_1}
    \right\}.
    \)

    \item[\upshape(4)]
    \(
    \pi=\mathds{1}_2 \times \tau,
    \)
    where \(\tau\in \Irr(\G_2)\) is an infinite-dimensional and 
    \(\Ha_{1,1}\)-distinguished.
\end{enumerate}
\end{prop}

\begin{proof}
          By Lemma~\ref{hari}, it is enough to consider representations of the form \(\pi=\pi_1\times\pi_2\), where either \(\pi_1\in\mathcal{C}(\G_1)\) and \(\pi_2\in\mathcal{C}(\G_3)\), or \(\pi_1,\pi_2\in\Irr(\G_2)\). We first note that, in the case \(\pi_1\in\mathcal{C}(\G_1)\) and \(\pi_2\in\mathcal{C}(\G_3)\), the representation \(\pi\) satisfies none of the conditions listed in Lemma~\ref{lem3}. Consequently, this case can not give rise to an \(\Ha_{1,3}\)-distinguished quotient. It remains to consider representations of the form \(\pi=\pi_1\times\pi_2\), with \(\pi_1,\pi_2\in\Irr(\G_2)\). As $\pi$ has an $\Ha_{1,3}$-distinguished quotient, it is also $\Ha_{1,3}$-distinguished. We divide the argument into five cases according to the possible types of \(\pi_1\) and \(\pi_2\), and in each case determine all representations \(\pi\) satisfying at least one of the conditions in Lemma~\ref{lem3} for the situation $n=4$ and $k=2$. We note that the cases corresponding to orbit (III) can be analyzed, according to the possible types of the irreducible representations \(\pi_1\) and \(\pi_2\), in a manner analogous to the cases summarized in Table~\ref{tab:placeholder_label} of Section~\ref{G_3}.

\noindent\textbf{Case (1).} We consider the following possibilities: (a) both \(\pi_1\) and \(\pi_2\) are supercuspidal; (b) neither of \(\pi_1\) and \(\pi_2\) is supercuspidal or a character; (c) \(\pi_1\) and \(\pi_2\) are not simultaneously twists of Steinberg representations; and (d) \(\pi_1\) and \(\pi_2\) are not simultaneously principal series representations. In all these situations, the conditions of Lemma~\ref{lem3} are not satisfied. Consequently, no \(\Ha_{1,3}\)-distinguished representation \(\pi\) arises in this case.

\smallskip
\noindent\textbf{Case (2).} Exactly one of \(\pi_1\) and
\(\pi_2\) is a character. By the conditions corresponding to the closed
orbits~(I) and~(II) in Lemma~\ref{lem3}, the only possible representations
are of the form \(\pi=\pi_1\times \mathds{1}_2\) or \(\pi=\mathds{1}_2\times \pi_2\), where \(\pi_1\) is \((\Ha_{1,1},\chi_{-1})\)-distinguished and \(\pi_2\) is \((\Ha_{1,1},\chi_{1})\)-distinguished. Since for an infinite-dimensional
irreducible representation of \(\G_2\), being
\((\Ha_{1,1},\chi_s)\)-distinguished is equivalent to being
\(\Ha_{1,1}\)-distinguished, this gives the fourth possibility
listed in the statement.

\smallskip
\noindent\textbf{Case (3).} Both \(\pi_1\) and \(\pi_2\) are
characters of \(\G_2\). Write \(\pi_1=\chi\circ\det\) and
\(\pi_2=\chi'\circ\det\). Equivalently, \(\pi_1\) is the unique quotient of
\(\nu^{\frac{1}{2}}_{\mathds{1}_1}\chi
\times \nu^{-\frac{1}{2}}_{\mathds{1}_1}\chi\), and \(\pi_2\) is the unique
quotient of
\(\nu^{\frac{1}{2}}_{\mathds{1}_1}\chi'
\times \nu^{-\frac{1}{2}}_{\mathds{1}_1}\chi'\). Applying
Lemma~\ref{lem3}, the distinction condition forces the inducing data to be
\(Z\!\left(\left[\frac{1}{2},\frac{3}{2}\right]_{(\mathds{1}_1)}\right)\)
and
\(Z\!\left(\left[-\frac{3}{2},-\frac{1}{2}\right]_{(\mathds{1}_1)}\right)\).
Hence, this yields the second possibility listed in the statement.

\smallskip
\noindent\textbf{Case (4).} Both \(\pi_1\) and \(\pi_2\) are
twist of Steinberg representations. Write \(\pi_1=\chi\St_2\) and
\(\pi_2=\chi'\St_2\). Equivalently, \(\pi_1\) is the unique quotient of
\(\nu^{-\frac{1}{2}}_{\mathds{1}_1}\chi
\times \nu^{\frac{1}{2}}_{\mathds{1}_1}\chi\), and similarly \(\pi_2\) is
the unique quotient of
\(\nu^{-\frac{1}{2}}_{\mathds{1}_1}\chi'
\times \nu^{\frac{1}{2}}_{\mathds{1}_1}\chi'\). By Lemma~\ref{lem3}, the
required distinction condition holds only when both twists are trivial.
Therefore, this gives the first possibility listed in the statement.

\smallskip
 \noindent\textbf{Case (5).} 
Both \(\pi_1\) and \(\pi_2\) are principal series representations of \(\G_2\). Write \(\pi_1=\sigma_1\times\sigma_2\) and \(\pi_2=\sigma_3\times\sigma_4\), where \(\sigma_i\in\Irr(\G_1)\). The conditions in Lemma~\ref{lem3} force $\pi$ to be
of the following forms:
\begin{itemize}
    \item \(\nu^{\frac{1}{2}}_{\mathds{1}_1}\times \mu \times \widetilde{\mu}\times \nu^{-\frac{1}{2}}_{\mathds{1}_1}\),
    \item \(\mu \times \nu^{\frac{1}{2}}_{\mathds{1}_1}\times  \widetilde{\mu}\times \nu^{-\frac{1}{2}}_{\mathds{1}_1}\),
    \item \(\nu^{\frac{1}{2}}_{\mathds{1}_1}\times \mu \times \nu^{-\frac{1}{2}}_{\mathds{1}_1} \times \widetilde{\mu}\),
    \item \(\mu \times \nu^{\frac{1}{2}}_{\mathds{1}_1}\times \nu^{-\frac{1}{2}}_{\mathds{1}_1} \times  \widetilde{\mu}\),
\end{itemize}
where $\mu \in \Irr(\G_1)$ with $\mu \not\in \left\{
\nu^{-\frac{1}{2}}_{\mathds{1}_1},
\nu^{\frac32}_{\mathds{1}_1}
\right\}$.
Since all possible forms of $\pi$ differ only by a permutation of the inducing factors, all have identical irreducible subquotients, and therefore it is sufficient to examine the subquotients of one of them. Thus, we restrict our attention to the situation \(
    \pi=
    \nu^{\frac{1}{2}}_{\mathds{1}_1}
    \times
    \mu
    \times
    \widetilde{\mu}
    \times
    \nu^{-\frac{1}{2}}_{\mathds{1}_1},
    \)
    where \(\mu\in\Irr(\G_1)\) with
    \(
    \mu \notin
    \left\{
    \nu^{-\frac{1}{2}}_{\mathds{1}_1},
    \nu^{\frac32}_{\mathds{1}_1}
    \right\}.
    \)
Hence, this yields the third possibility
listed in the statement.

Combining the above cases, we obtain precisely the four possibilities
listed in the statement. This completes the proof.
\end{proof}

We are now in a position to prove Conjecture~\ref{conj} in the case \(n=4\). More precisely, we establish the following theorem.

\begin{theorem}\label{t2}
      Let $\theta \in \Irr(\G_4)$. Then \(\theta\) is \(\Ha_{1,3}\)-distinguished iff it is either $\mathds{1}_4$ or of the form $\mathds{1}_{2} \times \tau$  where \(\tau\in \Irr(\G_2)\) is infinite-dimensional and \(\Ha_{1,1}\)-distinguished.
\end{theorem}

\begin{proof}
          The strategy of the proof is to analyze each representation appearing in Proposition~\ref{p2} and determine which of its irreducible subquotients, if any, are $\Ha_{1,3}$-distinguished. We divide the analysis into four cases. 

          \noindent \textbf{Case (1).} $\pi=\St_2\times \St_2$. In this case, \(\pi\) is irreducible and is the product of two discrete series representations. By Theorem~\ref{anand}, \(\pi\) is not \(\Ha_{1,3}\)-distinguished.

          \noindent \textbf{Case (2).} $\pi = Z\!\left(\left[\frac{1}{2},\frac{3}{2}\right]_{(\mathds{1}_1)}\right)\times Z\!\left(\left[-\frac{3}{2},-\frac{1}{2}\right]_{(\mathds{1}_1)}\right)$. In this case, $\pi$ has two irreducible subquotients $\theta_1 = Z\!\left(\left[-\frac{3}{2},\frac{3}{2}\right]_{(\mathds{1}_1)}\right)\; \text{and} \; \theta_2 = Z\!\left(\left[\frac{1}{2},\frac{3}{2}\right]_{(\mathds{1}_1)},\left[-\frac{3}{2},-\frac{1}{2}\right]_{(\mathds{1}_1)}\right)$ by Lemma~\ref{l1}. Since $\theta_1$ is the trivial representation, it is $\Ha_{1,3}$-distinguished. On the other hand, by the M{\oe}glin-Waldspurger algorithm \S \ref{walds}, $\theta_2$ is isomorphic to $L\!\left(\left[\frac{3}{2}\right]_{(\mathds{1}_1)},\left[-\frac{1}{2},\frac{1}{2}\right]_{(\mathds{1}_1)},\left[-\frac{3}{2}\right]_{(\mathds{1}_1)}\right)$,  which is unique irreducible quotient of $\nu^{\frac{3}{2}}_{\mathds{1}_1} \times L\!\left(\left[-\frac{1}{2},\frac{1}{2}\right]_{(\mathds{1}_1)}\right) \times \nu^{-\frac{3}{2}}_{\mathds{1}_1}$. Since $L\!\left(\left[-\frac{1}{2},\frac{1}{2}\right]_{(\mathds{1}_1)}, \left[-\frac{3}{2}\right]_{(\mathds{1}_1)}\right)$ is a quotient of $L\!\left(\left[-\frac{1}{2},\frac{1}{2}\right]_{(\mathds{1}_1)}\right) \times \nu^{-\frac{3}{2}}_{\mathds{1}_1}$, it follows that $\theta_2$ can be realized as the unique irreducible quotient of $\nu^{\frac{3}{2}}_{\mathds{1}_1} \times L\!\left(\left[-\frac{1}{2},\frac{1}{2}\right]_{(\mathds{1}_1)}, \left[-\frac{3}{2}\right]_{(\mathds{1}_1)}\right)$. 
          Applying Lemma \ref{lem3} with $n=4$ and $k=1$, we conclude that the latter representation is not $\Ha_{1,3}$-distinguished. 
          Consequently, $\theta_2$ is not $\Ha_{1,3}$-distinguished.

          \noindent \textbf{Case (3).} \label{Case3} $\pi=\nu^{\frac{1}{2}}_{\mathds{1}_1}\times\mu\times\widetilde{\mu}\times\nu^{-\frac{1}{2}}_{\mathds{1}_1}$, where  
          $\mu\in\Irr(\G_1)
          ~\text{with}~\ \mu \notin\left\{\nu^{-\frac{1}{2}}_{\mathds{1}_1},\nu^{\frac32}_{\mathds{1}_1}\right\}$. Here,
          if $\mu$ is neither $\nu^{-\frac{3}{2}}_{\mathds{1}_1}$ nor $\nu^{\frac{1}{2}}_{\mathds{1}_1}$, and is not of the form $\nu_{\sigma}^{\pm\frac{1}{2}}\sigma$ for any self-dual representation $\sigma$ with $\dim(\sigma)>1$, then by Lemma \ref{l2} we get that $\pi$ has two irreducible subquotients 
          $\theta_1 = \xi_1 \times \xi_2$ with $\xi_1 = \mathds{1}_{2}$ and $\xi_2 = \mu \times \widetilde{\mu}$ and
          $\theta_2 = \eta_1 \times \eta_2$ with $\eta_1 = \St_{2}$ and $\eta_2 = \mu \times \widetilde{\mu}$. 
          Within these, an $\Ha_{1,3}$-invariant linear form on $\theta_1$ arises from the closed orbit (I) by applying Lemma \ref{lem3} with $n=4$ and $k=2$. However, $\theta_2$ is not $\Ha_{1,3}$-distinguished because it 
          does not satisfy the Lemma \ref{lem3}
          for $n=4$ and $k=2$. Now assume otherwise. There are three subcases.

            \textbf{(3a)} $\mu=\nu_{\sigma}^{\pm\frac{1}{2}}\sigma$. In this case, by Lemma~\ref{l2}, $\pi$ has four irreducible subquotients  $\theta_1=\mathds{1}_2\times\St_2(\sigma)$, $\theta_2=\mathds{1}_2\times\Sp_2(\sigma)$, $\theta_3=\St_2\times\St_2(\sigma)$, and $\theta_4=\St_2\times\Sp_2(\sigma)$. Among these, $\theta_3$ and $\theta_4$ does not satisfy any of the conditions of Lemma~\ref{lem3} for $n=4$ and $k=2$, and hence are not $\Ha_{1,3}$-distinguished. On the other hand, $\theta_1$ and $\theta_2$ satisfy the condition of closed orbit (I) in Lemma \ref{lem3} with $n=4$ and $k=2$. Hence, using the characterization of $\Ha_{1,1}$-distinction for $\St_2(\sigma)$ and $\Sp_2(\sigma)$, we conclude that $\theta_1$ (resp.\ $\theta_2$) is $\Ha_{1,3}$-distinguished if $\omega_{\sigma}$ is non-trivial (resp. trivial).

          \textbf{(3b)} $\mu=\nu^{-\frac{3}{2}}_{\mathds{1}_1}$. In this situation, by Lemma \ref{l2}, $\pi$ has following eight irreducible subquotients:
        
\begin{align*}
\theta_{1} &= L\!\left(\left[-\tfrac{3}{2},\tfrac{3}{2}\right]_{(\mathds{1}_1)}\right), 
&\quad \theta_{2} &= L\!\left(\left[-\tfrac{3}{2},\tfrac{1}{2}\right]_{(\mathds{1}_1)},\left[\tfrac{3}{2}\right]_{(\mathds{1}_1)}\right), \\
\theta_{3} &= L\!\left(\left[-\tfrac{3}{2}\right]_{(\mathds{1}_1)},\left[-\tfrac{1}{2},\tfrac{3}{2}\right]_{(\mathds{1}_1)}\right),
&\quad \theta_{4} &= L\!\left(\left[\tfrac{1}{2},\tfrac{3}{2}\right]_{(\mathds{1}_1)},\left[-\tfrac{3}{2},-\tfrac{1}{2}\right]_{(\mathds{1}_1)}\right), \\
\theta_{5} &= L\!\left(\left[\tfrac{1}{2}\right]_{(\mathds{1}_1)},\left[-\tfrac{3}{2},-\tfrac{1}{2}\right]_{(\mathds{1}_1)},\left[\tfrac{3}{2}\right]_{(\mathds{1}_1)}\right),
&\quad \theta_{6} &= L\!\left(\left[-\tfrac{3}{2}\right]_{(\mathds{1}_1)},\left[-\tfrac{1}{2},\tfrac{1}{2}\right]_{(\mathds{1}_1)},\left[\tfrac{3}{2}\right]_{(\mathds{1}_1)}\right), \\
\theta_{7} &= L\!\left(\left[-\tfrac{3}{2}\right]_{(\mathds{1}_1)},\left[\tfrac{1}{2},\tfrac{3}{2}\right]_{(\mathds{1}_1)},\left[-\tfrac{1}{2}\right]_{(\mathds{1}_1)}\right),
&\quad \theta_{8} &= L\!\left(\left[\tfrac{1}{2}\right]_{(\mathds{1}_1)},\left[-\tfrac{3}{2}\right]_{(\mathds{1}_1)},\left[\tfrac{3}{2}\right]_{(\mathds{1}_1)},\left[-\tfrac{1}{2}\right]_{(\mathds{1}_1)}\right).
\end{align*}

      Next, we consider each subquotient separately. The question of $\Ha_{1,3}$-distinction of $\theta_6$ and $\theta_8$ has already been settled in Case $(2)$.  The representation $\theta_1$ is a discrete series, and therefore, it is not $\Ha_{1,3}$-distinguished by Theorem~\ref{anand}. Lemma \ref{l3} implies that $\theta_2$ is the unique irreducible quotient of $\nu_{(\mathds{1}_1)}^{\frac{3}{2}}\times L\!\left(\left[-\tfrac{3}{2},\tfrac{1}{2}\right]_{(\mathds{1}_1)}\right)$. Since the latter representation does not satisfy the conditions of Lemma~\ref{lem3} for $n=4$ and $k=1$, $\theta_2$ is not $\Ha_{1,3}$-distinguished. By taking contragredient, we further conclude that $\theta_3$ is also not $\Ha_{1,3}$-distinguished. The representation $\theta_4$ is the unique irreducible quotient of $L\!\left(\left[\tfrac{1}{2},\tfrac{3}{2}\right]_{(\mathds{1}_1)}\right)\times L\!\left(\left[-\tfrac{3}{2},-\tfrac{1}{2}\right]_{(\mathds{1}_1)}\right)$. The latter representation does not satisfy the conditions of Lemma~\ref{lem3}; hence, it is not $\Ha_{1,3}$-distinguished. Therefore, $\theta_4$ is also not $\Ha_{1,3}$-distinguished.

     Let us now analyze the subquotient $\theta_5$. By the M{\oe}glin--Waldspurger algorithm (see $\S$\ref{walds}), $\theta_5$ is isomorphic to $\theta_5'=Z\left(\left[-\frac{1}{2},\frac{3}{2}\right]_{(\mathds{1}_1)},\left[-\frac{3}{2}\right]_{(\mathds{1}_1)}\right).$ Moreover, it is easy to observe by Lemma \ref{l1} that $\theta_5'$ is the unique irreducible quotient of $\nu^{-\frac{3}{2}}_{\mathds{1}_1}\times Z\left(\left[-\frac{1}{2},\frac{3}{2}\right]_{(\mathds{1}_1)}\right)$. The latter representation is not $\Ha_{1,3}$-distinguished because it does not satisfy Lemma~\ref{lem3} for $n=4$ and $k=1$. So, $\theta_5'$ is not $\Ha_{1,3}$-distinguished, and therefore neither is $\theta_5$. It is easy to observe that $\widetilde{\theta_7}$ is isomorphic to $\theta_5$. Therefore, $\widetilde{\theta_7}$ is not $\Ha_{1,3}$-distinguished. Thus, $\theta_7$ is not $\Ha_{1,3}$-distinguished by Lemma~\ref{cont}.

\textbf{(3c)}  $\mu=\nu_{\mathds{1}_1}^{\frac{1}{2}}$. In this case, Lemma \ref{l4} shows that $\pi$ has three irreducible subquotients:

\begin{enumerate}

\item[] $\theta_{1}= L\!\left(\left[-\frac{1}{2},\frac{1}{2}\right]_{(\mathds{1}_1)}\right)\times L\!\left(\left[-\frac{1}{2},\frac{1}{2}\right]_{(\mathds{1}_1)}\right),$

\item[] $\theta_{2}= L\!\left(\left[\frac{1}{2}\right]_{(\mathds{1}_1)},\left[-\frac{1}{2}\right]_{(\mathds{1}_1)}\right)\times L\!\left(\left[-\frac{1}{2},\frac{1}{2}\right]_{(\mathds{1}_1)}\right),$

\item[] $\theta_{3}= L\!\left(\left[\frac{1}{2}\right]_{(\mathds{1}_1)},\left[-\frac{1}{2}\right]_{(\mathds{1}_1)}\right)\times L\!\left(\left[\frac{1}{2}\right]_{(\mathds{1}_1)},\left[-\frac{1}{2}\right]_{(\mathds{1}_1)}\right).$

\end{enumerate}
The Case (1) covers $\Ha_{1,3}$ distinction of $\theta_1$. As $\theta_2$ satisfies condition of  closed orbit (I) in Lemma \ref{lem3} with $n=4$ and $k=2$, it is $\Ha_{1,3}$-distinguished. However, $\theta_3$ is not $\Ha_{1,3}$-distinguished since it fails to satisfy Lemma \ref{lem3} for $n=4$ and $k=2$.

       \noindent \textbf{Case (4).} \(
    \pi=\mathds{1}_2 \times \tau,
    \)
    where \(\tau\in \Irr(\G_2)\) is an infinite-dimensional and 
    \(\Ha_{1,1}\)-distinguished.
 In this context, the analysis naturally divides into the following subcases according to the structure of \(\tau\).

\textbf{(4a)} Suppose that \(\tau\) is not isomorphic to
\(\nu_{\mathds{1}_1}^{\frac{3}{2}}
\times
\nu_{\mathds{1}_1}^{-\frac{3}{2}}\).
In this case, the representation \(\pi\) is irreducible. Applying Lemma~\ref{lem3} with \(n=4\) and \(k=2\), the closed orbit (I) yields an \(\Ha_{1,3}\)-invariant linear form on \(\pi\).

\textbf{(4b)}  Suppose that \(\tau\) is isomorphic to
\(
\nu_{\mathds{1}_1}^{\frac{3}{2}}
\times
\nu_{\mathds{1}_1}^{-\frac{3}{2}}.
\)
In this scenario, Lemma~\ref{l2} shows that \(\pi\) is reducible and has precisely four irreducible subquotients, namely
\[
\begin{aligned}
\theta_1
&= Z\!\left(\left[-\tfrac{3}{2},\tfrac{3}{2}\right]_{(\mathds{1}_1)}\right), 
&
\theta_2
&= Z\!\left(\left[-\tfrac{3}{2},\tfrac{1}{2}\right]_{(\mathds{1}_1)},
\left[\tfrac{3}{2}\right]_{(\mathds{1}_1)}\right), \\[4pt]
\theta_3
&= Z\!\left(\left[-\tfrac{1}{2},\tfrac{3}{2}\right]_{(\mathds{1}_1)},
\left[-\tfrac{3}{2}\right]_{(\mathds{1}_1)}\right), 
&
\theta_4
&= Z\!\left(\left[-\tfrac{1}{2},\tfrac{1}{2}\right]_{(\mathds{1}_1)},
\left[\tfrac{3}{2}\right]_{(\mathds{1}_1)},
\left[-\tfrac{3}{2}\right]_{(\mathds{1}_1)}\right).
\end{aligned}
\]

         \noindent The question of $\Ha_{1,3}$-distinction of these subquotients has already been settled in the previous cases.

         We now establish the converse direction of the theorem. If $\theta$ is trivial, then it is evidently $\Ha_{1,3}$-distinguished. In the situation where $\theta=\mathds{1}_2\times \tau$, with $\tau\in\Irr(\G_2)$ infinite-dimensional and $\Ha_{1,1}$-distinguished, Lemma~3.2 applied with $n=4$ and $k=2$ shows that $\theta$ admits $\Ha_{1,3}$-distinction arising from the closed orbit (I).
\end{proof}

\section{Multiplicity-one property}\label{sec5}
        In this section, we prove Corollary~\ref{corr}. The proof is a direct consequence of the description of $\Ha_{1,n-1}$-distinguished representations given in Conjecture \ref{conj}.

\begin{proof}

           The validity of the corollary for $n\leq 3$ is already known. The case $n=1$ is trivial, whereas the cases $n=2$ and $n=3$ were established in \cite[Theorem~6.1]{raghukiri} and \cite[Theorem~2.5]{anandavardhanan2024sign}, respectively. So, let us assume $n \geq 4$. Assuming Conjecture~\ref{conj}, every irreducible $\Ha_{1,n-1}$-distinguished representation $\theta$ of $\G_n$ is either the trivial representation $\mathds{1}_n$ or of the form $\mathds{1}_{n-2}\times\tau$, where $\tau$ is an infinite-dimensional irreducible representation of $\G_2$. In the case where $\theta=\mathds{1}_n$, it is obvious that $\dim\Hom_{\Ha_{1,n-1}}(\theta,\mathbb{C})=1$. Now suppose that $\theta=\mathds{1}_{n-2}\times\tau$. By applying Lemma~\ref{lem3} with $k=n-2$, it is easy to observe that 
           $$\Hom_{\Ha_{1,n-1}}(\theta,\mathds{1})\simeq
           \Hom_{\G_{n-2}}(\mathds{1}_{n-2},\mathds{1}_{n-2})\otimes
           \Hom_{\Ha_{1,1}}(\tau,\mathds{1})\simeq
           \Hom_{\Ha_{1,1}}(\tau,\mathds{1}).$$
           Thus, the desired conclusion follows from the multiplicity-one theorem of Prasad--Raghuram~\cite[Theorem~6.1]{raghukiri}.

\end{proof}

\subsection*{Acknowledgement}
The authors thank Dipendra Prasad for his valuable comments and insightful
 suggestions.
 
\bibliographystyle{alpha}
\bibliography{linear}		
\end{document}